\documentclass[12pt,reqno]{amsart}
\usepackage{amsmath,amssymb,amsthm,graphicx,url}

\newtheorem{lemm}{Lemma}
\newtheorem{defi}{Definition}
\newtheorem{theo}{Theorem}
\theoremstyle{definition}

\DeclareMathOperator{\Imag}{Im}
\DeclareMathOperator{\Real}{Re}

\title[Symmetry of bound and antibound states]
{Symmetry of bound and antibound states\\
in the semiclassical limit for a general class
of potentials}
\author[S. Dyatlov]
{Semyon Dyatlov}
\email{dyatlov@math.berkeley.edu}
\author[S. Ghosh]
{Subhroshekhar Ghosh}
\email{subhro@math.berkeley.edu}
\address{Mathematics Department, University of California \\
Evans Hall, Berkeley, CA 94720, USA}

\begin{document}

\begin{abstract}
We consider the Schr\"odinger operator $-h^2\partial_x^2+V(x)$ on a
half-line, where $V$ is a compactly supported potential which is
positive near the endpoint of its support.  We prove that the
eigenvalues and the purely imaginary resonances are symmetric up~to an
error $Ce^{-\delta/h}$.
\end{abstract}

\maketitle

\section{Introduction}

In this paper, we study spectral properties of the Schr\"odinger operator
$$
P(h)=-h^2 \partial_x^2+V(x)
$$
defined for $x$ in the half-line $(-\infty, B]$. Here $h>0$ is the
semiclassical parameter and $V(x)$ is a piecewise continuous
real-valued potential supported in $[0,B]$.

The operator $P(h)$ with the Neumann boundary condition at $B$ is
self-adjoint on $L^2(-\infty, B)$; therefore, its resolvent
$$
R_V(\lambda)=(P(h)-\lambda^2)^{-1},\
\Imag\lambda>0,
$$
is a bounded operator from $L^2$ to $H^2$ for $\lambda^2$ not in the
spectrum of $P(h)$. This resolvent can be extended meromorphically as
an operator $L^2_{\textrm{comp}}\to H^2_{\textrm{loc}}$ to $\lambda\in
\mathbb C$ with isolated poles of finite rank; these poles are called
\textbf{resonances}. (The reader is referred to \cite{TZ} for
details.) To each resonance $\lambda$ corresponds a \textbf{resonant
state}; that is, a nonzero $u\in H^2_{\textrm{loc}}(-\infty, B)$
solving the equation $(P(h)-\lambda^2)u=0$ with the Neumann boundary
condition at the right endpoint and with the following
\textbf{outgoing condition} at $-\infty$:
$$
u(x)=Ae^{-{i\lambda x/h}}\text{ for all }x<0\text{ and some constant }A.
$$
(Note that for $x<0$, $u$ solves the free equation $(-h^2
\partial_x^2-\lambda^2)u=0$, so it must be a linear combination of
$e^{\pm {i\lambda x/h}}$.)

For $\Imag\lambda>0$, the outgoing condition implies that $u$ is
exponentially decreasing on the negative half-line and thus $u\in
L^2$; therefore, $\lambda$ is a pole (of the resolvent) lying in the
upper half-plane if and only if $\lambda^2$ is an eigenvalue of $P(h)$
on $L^2$. Since $P(h)$ is self-adjoint, all poles in the upper
half-plane have to lie on the imaginary axis. There may be poles
$\lambda$ with $\Imag\lambda<0$ and $\Real\lambda\neq 0$; however, we
will restrict our attention to purely imaginary resonances:
\begin{defi}
A positive number $k$ is called a \textbf{bound state} if $ik$ is a
pole of the resolvent $R_V(\lambda)$, and an \textbf{antibound state}
if $-ik$ is a pole.
\end{defi}

We see from above that $k$ is an (anti)bound state if and only if
there exists a nonzero solution $u$ of the problem
\begin{gather}
\label{e:equation}
(P(h)+k^2)u=0,\\
\label{e:rightcond}
u_x|_{x=B}=0,\\
\label{e:leftcondo}
hu_x\pm ku|_{x=0}=0.
\end{gather}
The plus sign in (\ref{e:leftcondo}) corresponds to an antibound state
and the minus sign corresponds to a bound state. We will also study
Neumann eigenvalues of $P(h)$ on $[0,B]$; i.e., those $k$ for which
there exists a nonzero solution to (\ref{e:equation}) with boundary
conditions (\ref{e:rightcond}) and
\begin{equation}\label{e:leftcondn}
u_x|_{x=0}=0.
\end{equation}

Since the space of solutions to~(\ref{e:equation}) and~(\ref{e:rightcond})
is always one dimensional, \textbf{bound states, antibound states, and Neumann
eigenvalues never coincide}. However, Bindel and Zworski proved in \cite{BZ}
that bound and antibound states located away from zero coincide modulo
errors of order $e^{-\delta/h}$ for some $\delta>0$, if the potential
satisfies the following conditions:
$$
\begin{gathered} 
\exists A>0,V_0>0:V(x)=V_0\text{ for all }x\in [0,A],\\
\exists \varepsilon>0:V(x)=0\text{ for all }x\in (A,A+\varepsilon).
\end{gathered}
$$
In this paper, we prove a similar result with more general assumptions on the potential:

\begin{theo}\label{l:main}
Suppose that $V$ is a piecewise continuous real-valued potential
supported in $[0,B]$ and satisfying the following \textbf{bump
condition}:
\begin{equation}\label{e:bump}
\exists A>0: V(x)>0\text{ for all }x\in (0,A].
\end{equation}
Fix two constants $0<c_k<C_k<\infty$. Then there exist constants
$C,\delta>0,h_0>0$ such that for $h<h_0$ and any $k\in [c_k,C_k]$:

1. If $k$ is a Neumann eigenvalue, then there exist a bound state $k_+$
and an antibound state $k_-$ such that $|k-k_\pm|\leq Ce^{-\delta/h}$.

2. If $k$ is a bound or an antibound state, then there exists a Neumann eigenvalue
$k_0$ such that $|k-k_0|\leq Ce^{-\delta/h}$.
\end{theo}

\begin{figure}
\includegraphics{plots.1}
\caption{Bound and antibound states for two spline potentials
(\texttt{splinepot([0, -0.4, -1, -0.2, -1, -0.4, 0], [-2, -1.5, -1, 0, 1, 1.5, 2])}
and
\texttt{splinepot([0, 0.2, -1, -0.2, -1, 0.2, 0], [-2, 1.5, -1, 0, 1, 1.5, 2])})}
\end{figure}

The bump condition (\ref{e:bump}) cannot be disposed of completely, as
illustrated by the numerical experiments performed using \cite{B}.
Figure~1 shows two potentials on the whole line, each supported in
$[-2,2]$, and the corresponding bound states (denoted by squares) and
antibound states (denoted by circles).  The vertical coordinate of
each (anti)bound state on the picture corresponds to its value $k$;
the horizontal coordinate corresponds to the value of $h^{-1}$
used. We see that the conclusion of the theorem does not appear to
hold for the potential on the left, which does not satify the bump
condition; at the same time, it is true for the potential on the
right. Theorem~1, formulated for the half-line case, applies to these
numerical experiments on the whole line since for even potentials, the
set of their (anti)bound states is composed of these states for the
positive half-line with Dirichlet condition and the same states for
the Neumann condition; the theorem above can be applied with Dirichlet
condition in place of (\ref{e:rightcond}).  (However, condition
(\ref{e:leftcondn}) cannot be replaced by Dirichlet condition in the
theorem.)

The study of resonances in one dimension has a long tradition going
back to origins of quantum mechanics, see for instance~\cite{LL}.  One
of the first studies of their distribution was conducted by
Regge~\cite{R}; since then, there have been many mathematical results
on the topic, including~\cite{AA}, \cite{BC}, \cite{F}, \cite{H},
\cite{K}, \cite{N}, \cite{S}, and~\cite{Z}.  Concerning antibound
states, Hitrik has shown in~\cite{H} that for a positive compactly
supported potential, there are no antibound states in the
semiclassical limit. This agrees with our result since there are no
bound states in this case. Simon proved in~\cite{S} that between any
two bound states, there must be an odd number of antibound states; the
following corollary of this result follows almost immediately using
the methods we develop to prove Theorem~1:

\begin{theo}\label{l:simon}
Consider the half-line problem with a bounded compactly supported
potential $V$ (which does not need to satisfy any positivity
condition). Then for each two bound states $0<k_1<k_2$, the interval
$(k_1,k_2)$ contains at least one antibound state.  In particular, if
there are $n$ bound states in some subinterval of $(0,\infty)$, then
there are at least $n-1$ antibound states in the same subinterval.
\end{theo} 

The proof of Theorem 1 works as follows: we study the evolution (in $x$)
of the
vectors $(u,hu_x)$ for the three solutions of (\ref{e:equation}) with
initial data at $x=0$ satisfying the conditions (\ref{e:leftcondo})
and (\ref{e:leftcondn}).  The idea is to look at these three vectors
at $x=A$. Since $V(x)+k^2\geq 0$ on the interval $(0,A)$, the
transition matrix for the considered vectors from $x=0$ to $x=A$ will
have an expanding and a contracting direction.  (In fact, if we
introduce rescaling $\tilde x=x/h$, then the behavior of the original
system for small $h$ is similar to the behavior of the rescaled system
for large $\tilde x$, and the latter will be similar to the behavior
of the geodesic flow on a two-dimensional manifold of negative
curvature.)  It turns out that our three vectors lie in a certain
angle between the expanding and the contracting directions, from which
it follows that they will stay in this angle for later times
(Lemma~\ref{l:estimate}); what is more, their polar angles will get
exponentially close to each other (Lemma~\ref{l:close}).  Finally, we
can study how the polar angles of the considered vectors change with
$k$ (Lemma~\ref{l:angle}): it follows (Lemma~\ref{l:nondegenerate})
that the polar angle for the solution with Neumann initial data at
$x=0$ will strictly increase in $k$ and the polar angle for the
solution with the same data at $x=B$ will decrease in $k$. The proof
is then completed by a pertrubation argument (Lemma~\ref{l:perturb}).

The detailed proofs of Theorems~\ref{l:main} and~\ref{l:simon} are given in
Section~3.  Both are elementary and use certain properties of ordinary
differential equations presented in Section~2.

The authors would like to thank Maciej Zworski for suggesting the
problem and for many illuminating discussions.

\section{Preliminaries}

Throughout this section, $I$ is an interval in $\mathbb R$, $V(x)\in
L^\infty(I;\mathbb R)$, $u(x), v(x)\in H^2(I;\mathbb R)$, $h>0$, and
$P(h)=-h^2\partial_x^2+V(x)$. 
Any solution to the equation $P(h)u=0$
is determined by the vector $(u,hu_x)$ at any $x$; we will sometimes
view this vector in polar coordinates: 

\begin{defi}\label{d:langle}
Define the \textbf{length} $L(u)$ and the \textbf{polar angle}
$\theta(u)$ by the equations
$$
\begin{gathered}
u=L(u)\cos\theta(u),\\
hu_x=L(u)\sin\theta(u).
\end{gathered}
$$
Here $\theta(u)$ lies in the circle $\mathbb S^1=\mathbb R/2\pi\mathbb Z$.
\end{defi}

\begin{lemm}\label{l:wronskian} Define the \textbf{Wronskian} $W(u,v)$ by
$$
W(u,v)=h(uv_x-vu_x).
$$
Then 
\begin{gather}
\label{e:wrondef}W(u,v)=L(u)L(v)\sin(\theta(v)-\theta(u)),\\
\label{e:wronskian}h \partial_x W(u,v)=v\cdot P(h)u-u\cdot P(h)v.
\end{gather}
\end{lemm}
Note that the $W(u,v)$ is just the oriented area of the parallelogram
spanned by the vectors $(u,hu_x)$ and $(v,hv_x)$. The next lemma tells
us that if the vector $(u,hu_x)$ falls inside a certain angle in the
plane at the initial time, then it will stay inside that angle for all
later times:

\begin{lemm}\label{l:estimate}
Suppose that $a^2\leq V(x)\leq b^2$ for all $x\in I$ and some
constants $a,b>0$. Let $u$ be a solution to $P(h)u=0$ and define
$$
W_+(u)=W(u,e^{bx/ h}),\
W_-(u)=W(e^{-{ax/h}},u).
$$
Let $x_0$ be a point in $I$ and assume that $W_+(u),W_-(u)\geq 0$ at
$x_0$.  Then for $x\geq x_0$, the functions $W_\pm(u)$ are increasing in $x$
and
\begin{equation}\label{e:ugeq}
u\geq \frac{L(u)}{\sqrt{1+b^2}}.
\end{equation}
\end{lemm}
\begin{proof}
We have
$$
e^{-{bx/ h}}W_+(u)=bu-hu_x,\
e^{ax/h}W_-(u)=au+hu_x.
$$
Therefore, $W_+(u),W_-(u)\geq 0$ yields $|hu_x|\leq bu$ and thus
(\ref{e:ugeq}). Next,
$$
\begin{gathered} 
P(h)e^{bx/ h}=e^{bx/ h}(V(x)-b^2)\leq 0,\\
P(h)e^{-{ax/ h}}=e^{-{ax/ h}}(V(x)-a^2)\geq 0.
\end{gathered}
$$
Using (\ref{e:wronskian}), we see that $\partial_xW_\pm\geq 0$ as long
as $u\geq 0$. It remains to prove that $u(x)\geq 0$ for $x\geq x_0$.
Suppose this is false and let $x_1=\inf\{x\geq x_0\mid u(x)<0\}$.
Then $u$ is not identically zero; since it solves a second
order linear ODE, $L(u)>0$ everywhere.  But
$u\geq 0$ on $[x_0,x_1]$, so $W_\pm$ are increasing on this interval.
In particular, $W_\pm\geq 0$ at $x_1$ and thus (\ref{e:ugeq}) holds at
this point.  However, by the choice of $x_1$ we have $u(x_1)=0$, which
contradicts $L(u)>0$.
\end{proof}

In the next section, we will use the following crude estimate
on how fast the solutions of an ODE can grow:
\begin{lemm}\label{l:crude}
Assume that $|V(x)|\leq M$ for $x\in I$ and that $u$ is a solution to
$P(h)u=0$. Take $x_0,x_1\in I$; then
$$
L(u)|_{x=x_1}\leq e^{(1+M)|x_0-x_1|/(2h)}\cdot L(u)|_{x=x_0}.
$$
\end{lemm}
\begin{proof}
Without loss of generality we may assume that $x_1>x_0$.  We have
$L(u)^2=u^2+(hu_x)^2$; thus
$$
h \partial_x(L(u)^2)=2huu_x(1+V(x))\leq (1+M)L(u)^2
$$
and the lemma is proven by Gronwall's inequality.
\end{proof}

\begin{lemm}\label{l:angle}
Assume that $u(x,k)$ is a family of solutions to $(P(h)+k^2)u=0$,
$x_0,x_1\in I$, and $u(x_0,k)$ and $u_x(x_0,k)$ are independent of
$k$.  Let $\Theta_1(k)=\theta(u(x,k))|_{x=x_1}$,
$L_1(k)=L(u(x,k))|_{x=x_1}$. Then
$$
\Theta_1'(k)=\frac{2k}{hL_1(k)^2}\int_{x_0}^{x_1}u(x,k)^2\,dx.
$$
\end{lemm}
\begin{proof}
We have $W(u,u_k)|_{x=x_1}=L_1(k)^2\Theta'_1$.
(To see that, differentiate the formulas in Definition~\ref{d:langle}
in $k$ and use the definition of the Wronskian.)
Now, we differentiate
the equation $(P(h)+k^2)u=0$ in $k$ to get $(P(h)+k^2)u_k=-2ku$.  It
remains to apply (\ref{e:wronskian}) together with
$W(u,u_k)|_{x=x_0}=0$.
\end{proof}

\begin{lemm}\label{l:perturb}
Assume that $\Phi$ is a $C^1$ map from the interval
$I=[K_0,K_1]$ to the circle $S^1=\mathbb R/2\pi \mathbb Z$
and $\Phi'(k)\geq\delta>0$ for all $k\in
I$.  Suppose that $\Psi:I\to \mathbb S^1$ is a continuous map such
that $|\Psi(k)|\leq\varepsilon<\pi$ for all $k$. Put
$\nu=\varepsilon/\delta$ and $I_\nu=[K_0+\nu,K_1-\nu]$. Then:

1. If $k_0\in I_\nu$ has $\Phi(k_0)=0$, then there exists $k_1\in I$
with $\Phi(k_1)=\Psi(k_1)$ and $|k_0-k_1|\leq\nu$.

2. If $k_1\in I_\nu$ has $\Phi(k_1)=\Psi(k_1)$, then there exists
$k_0\in I$ with $\Phi(k_0)=0$ and $|k_0-k_1|\leq\nu$.
\end{lemm}
\begin{proof}
We can lift $\Phi$ and $\Psi$ to continuous maps
$\bar\Phi,\bar\Psi:I\to \mathbb R$; then $|\bar\Psi|\leq\varepsilon$
and $\bar\Phi(k')-\bar\Phi(k)\geq \delta(k'-k)$ for $k'\geq k$.

1. We have $\bar\Phi(k_0)=2\pi m$ for some $m\in \mathbb Z$.  Then
$\bar\Phi(k_0+\nu)\geq 2\pi m+\delta\nu\geq 2\pi m+\bar\Psi(k_0+\nu)$
and $\bar\Phi(k_0-\nu)\leq 2\pi m+\bar\Psi(k_0-\nu)$; it remains to
apply the intermediate value theorem.

2. Similar to the previous statement, we have $\bar\Phi(k_1)=2\pi
m+\bar\Psi(k_1)$ for some $m\in \mathbb Z$ and $\bar\Phi(k_1+\nu)\geq
2\pi m\geq \bar\Phi(k_1-\nu)$.
\end{proof}

\begin{lemm}\label{l:calc}
Assume that $\Phi$ is a $C^1$ map from some interval $I$ to the circle
$\mathbb S^1=\mathbb R/(2\pi \mathbb Z)$ with $\Phi'(k)>0$ for all
$k\in I$. Let $\Psi:I\to \mathbb S^1$ be a continuous map such that
$\Psi(k)\neq 0$ for all $k\in I$.  If $k_1<k_2$ are two roots of the
equation $\Phi=0$, then the interval $(k_1,k_2)$ contains at least one
root of the equation $\Phi=\Psi$.
\end{lemm}
\begin{proof}
As in the previous lemma, lift $\Phi$ and $\Psi$ to maps
$\bar\Phi,\bar\Psi:I\to \mathbb R$; we can make $0<\bar\Psi(k)<2\pi$
for all $k\in I$. Since $\bar\Phi'>0$ everywhere, we have
$\bar\Phi(k_j)=2\pi m_j$, where $m_1<m_2$ are some integers.  Therefore,
$\bar\Phi(k_1)<2\pi m_1+\bar\Psi(k_1)$ and $\bar\Phi(k_2)>2\pi
m_1+\bar\Psi(k_2)$; it remains to apply the intermediate value
theorem.
\end{proof}

\section{Proofs of the theorems}

We assume in this section that $0<c'_k\leq k\leq C'_k$ for some
constants $c'_k<c_k$ and $C'_k>C_k$; the constants in our estimates
will depend on $c'_k$ and $C'_k$. (We need to expand the interval
$[c_k,C_k]$ a little bit to be able to apply Lemma~\ref{l:perturb}.)

Consider the solutions $u_\pm,u_0,u_1(x,k)$ to the equation
(\ref{e:equation}) in $[0,B]$ with the initial data
$$
\begin{gathered}
u_{\pm 0}(0,k)=1,\
\partial_x u_0(0,k)=0,\
h\partial_xu_\pm(0,k)=\pm k,\\
u_1(B,k)=1,\
\partial_xu_1(B,k)=0.
\end{gathered}
$$
Define $\Theta_0(k)$, $\Theta_\pm(k)$, and $\Theta_1(k)$ to be the
polar angles of vectors $(u,hu_x)$ at $x=A$ for
$u=u_0,u_\pm,u_1$. Then $k>0$ is

\begin{itemize}
\item a Neumann eigenvalue if $u_0$ and $u_1$ are linearly dependent; that is,
(recalling that they solve the same second order ODE) if $2(\Theta_0(k)-\Theta_1(k))=0$;
\item a bound state if $2(\Theta_+(k)-\Theta_1(k))=0$;
\item an antibound state if $2(\Theta_-(k)-\Theta_1(k))=0$.
\end{itemize}
Here we count angles modulo $2\pi$.

To prove Theorem~\ref{l:main}, it suffices to use Lemma~\ref{l:perturb} (for
$\Phi=2(\Theta_0-\Theta_1)$ and $\Psi=2(\Theta_0-\Theta_\pm)$)
together with the following two facts:

\begin{lemm}\label{l:close}
For some constants $C_1$ and $\delta_1>0$ independent of $h$ and $k$,
$$
|2(\Theta_0(k)-\Theta_\pm(k))|\leq C_1e^{-{\delta_1/ h}}\text{ for all }k\in[c'_k,C'_k].
$$
\end{lemm}

\begin{lemm}\label{l:nondegenerate}
We have $\Theta_0'(k)-\Theta_1'(k)\geq {1/ C_2}> 0$ for all $k\in [c'_k,C'_k]$ and
some constant $C_2$ independent of $h$ and $k$.
\end{lemm}

We first prove Lemma~\ref{l:close}. Put $b=\max_{[0,A]}V(x)$,
$k_b=\sqrt{k^2+b}$, $\psi_0(x)=e^{-{kx/ h}}$,
$\psi_+(x)=e^{k_bx/ h}$, and consider the Wronskians
$$
W_+(u)=W(u,\psi_+),\
W_0(u)=W(\psi_0,u).
$$
These are nonnegative for $u=u_0,u_\pm$ at $x=0$.  Then by
Lemma~\ref{l:estimate}, all these six functions are nonnegative and
increasing in $x$ for $0\leq x\leq A$.

Our first goal is to get an exponential lower bound on the length
$L(u)$ for $u=u_0,u_\pm$ at $x=A$. For $u_0$, note that by
(\ref{e:wrondef})
$$
L(u_0)\geq \frac{W(\psi_0,u_0)}{L(\psi_0)}
\geq \frac {W_0(u_0)|_{x=0}}{L(\psi_0)}
\geq\frac 1Ce^{kx/h}.
$$
Same applies to $u_+$. However, $u_-$ needs more careful analysis
since $W_0(u_-)=0$ at $x=0$. For that, take $0<t<1$ and put
$a=\min_{[t A,A]}V(x)>0$, $k_a=\sqrt{k^2+a}$, $\psi_-(x)=e^{-{k_ax/
h}}$, and $W_-(u)=W(\psi_-,u)$. First, we have by Lemma~\ref{l:crude}
$$
L(u_-)\geq e^{-(1+k^2+b)x/(2h)}\cdot L(u_-)|_{x=0}.
$$
Next, $W_0(u_-)\geq 0$ and $W_+(u_-)\geq 0$, so by~(\ref{e:ugeq})
$$
W_-(u_-)\geq (k_a-k)u_-\psi_-\geq \frac 1C L(u_-)\psi_-.
$$
Finally, we apply Lemma~\ref{l:estimate} on the interval $[t A,A]$ to get
$$
L(u_-)|_{x=A}\geq \frac{W_-(u_-)|_{x=tA}}{L(\psi_-)|_{x=A}}
\geq \frac 1Ce^{(k(1-t)-(1+k^2+b)t)A/ h}.
$$
For $t$ small enough and all $k$, $k(1-t)-(1+k^2+b)t\geq 0$,
so we have
$$
L(u_-)|_{x=A}\geq \frac 1C>0.
$$

The next step is to use that $u_0$ and $u_\pm$ solve the same equation
(\ref{e:equation}) and thus $W(u_0,u_\pm)$ is constant in
$x$. Therefore, at $x=A$ we have by~(\ref{e:wrondef})
$$
|\sin(\theta(u_\pm)-\theta(u_0))|=\frac{|W(u_0,u_\pm)|}{L(u_0)L(u_\pm)}\leq Ce^{-{kA/ h}}.
$$
That finishes the proof of Lemma~\ref{l:close}.

To prove Lemma~\ref{l:nondegenerate}, first note that by
Lemma~\ref{l:angle}, $\Theta_1'(k)\leq 0$ and
$$
\Theta_0'(k)\geq \frac 1{ChL(u_0)^2|_{x=A}}\int_0^A |u_0(x,k)|^2\,dx
$$
By (\ref{e:ugeq}), $u_0\geq L(u_0)/C$. Also, by Lemma~\ref{l:crude},
$L(u_0)\geq e^{C(x-A)/ h}L(u_0)|_{x=A}$ for $0\leq x\leq A$; thus
$$
\int_0^A |u_0(x,k)|^2\,dx\geq \frac 1C\int_0^A e^{C(x-A)/ h}(L(u_0)^2|_{x=A})\,dx
\geq \frac hCL(u_0)^2|_{x=A}
$$
and Lemma~\ref{l:nondegenerate} is proven, which finishes the proof of
Theorem~\ref{l:main}.

\smallskip

To prove Theorem~\ref{l:simon}, let
$\Phi_\pm(k)=\theta(u_\pm)|_{x=B}$; a bound state corresponds to
$2\Phi_+=0$ and an antibound state corresponds to $2\Phi_-=0$.  Since
$\theta(u_+)|_{x=0}$ is increasing with $k$, by an argument similar to
the proof of Lemma~\ref{l:angle} we get $\Phi'_+(k)>0$ for all
$k$. Moreover, $2(\Phi_+(k)-\Phi_-(k))$ is never zero, as this would
correspond to $u_+$ and $u_-$ being linearly dependent.  We may now
apply Lemma~\ref{l:calc} with $\Phi=2\Phi_+$ and
$\Psi=2(\Phi_+-\Phi_-)$.

\end{document}